\documentclass[11pt]{article}
\usepackage{geometry}                
\usepackage{graphicx}
\usepackage{amsmath, amssymb, amsthm}
\usepackage{lscape}
\usepackage{multirow}
\usepackage{url}

\newtheorem{thm}{Theorem}
\newtheorem{lem}{Lemma}

\title{Piecewise linear approximations of the standard normal first order loss function}
\author{Roberto Rossi,${}^{*,1}$ S. Armagan Tarim,${}^2$ Steven Prestwich,${}^3$ Brahim Hnich${}^4$\\
${}^1$Business School, University of Edinburgh, Edinburgh, UK\\
roberto.rossi@ed.ac.uk\\
${}^2$Dept. of Management, Hacettepe University, Ankara, Turkey\\
armagan.tarim@hacettepe.edu.tr\\
${}^3$Dept. of Computer Science, University College Cork, Cork, Ireland\\
s.prestwich@cs.ucc.ie\\
${}^4$Dept. of Computer Engineering, Izmir University of Economics, Izmir, Turkey\\
hnich.brahim@gmail.com
}
\date{}                                           

\begin{document}
\maketitle

\begin{abstract}
The first order loss function and its complementary function are extensively used in practical settings. When the random variable of interest is normally distributed, the first order loss function can be easily expressed in terms of the standard normal cumulative distribution and probability density function. However, the standard normal cumulative distribution does not admit a closed form solution and cannot be easily linearised. Several works in the literature discuss approximations for either the standard normal cumulative distribution or the first order loss function and their inverse. However, a comprehensive study on piecewise linear upper and lower bounds for the first order loss function is still missing. In this work, we initially summarise a number of distribution independent results for the first order loss function and its complementary function. We then extend this discussion by focusing first on random variable featuring a symmetric distribution, and then on normally distributed random variables. For the latter, we develop effective piecewise linear upper and lower bounds that can be immediately embedded in MILP models. These linearisations rely on constant parameters that are independent of the mean and standard deviation of the normal distribution of interest. We finally discuss how to compute optimal linearisation parameters that minimise the maximum approximation error.\\
{\bf keywords: } first order loss function; complementary first order loss function; piecewise linear approximation; minimax; Jensen's; Edmundson-Madansky.\\
{\bf Corresponding author}: Roberto Rossi, University of Edinburgh Business School, EH8 9JS, Edinburgh, United Kingdom.\\ phone: +44(0)131 6515239\\ email: roberto.rossi@ed.ac.uk
\end{abstract}

\newpage

\section{Introduction}

Consider a random variable $\omega$ and a scalar variable $x$. The first order loss function is defined as
\begin{equation}
\mathcal{L}(x,\omega)=\mbox{E}[\max(\omega-x,0)],
\end{equation}
where $\mbox{E}$ denotes the expected value.
The complementary first order loss function is defined as
\begin{equation}
\widehat{\mathcal{L}}(x,\omega)=\mbox{E}[\max(x-\omega,0)].
\end{equation}
The first order loss function and its complementary function play a key role in several application domains. In inventory control \cite{spp98} it is often used to express expected inventory holding or shortage costs, as well as service level measures such as the widely adopted ``fill rate'', also known as $\beta$ service level \cite{Axsater2006}, p. 94. In  finance the first order loss function may be employed to capture risk measures such as the so-called ``conditional value at risk'' (see e.g. \cite{citeulike:1194520}). These examples illustrate possible applications of this function. Of course, the applicability of this function goes beyond inventory theory and finance. 

Despite its importance, to the best of our knowledge a comprehensive analysis of results concerning the first order loss function seems to be missing in the literature. In Section \ref{sec:loss_function}, we first summarise a number of distribution independent results for the first order loss function and its complementary function. We then focus on symmetric distributions and on normal distributions; for these we discuss ad-hoc results in Section \ref{sec:loss_function_normal}.

According to one of these results, the first order loss function can be expressed in terms of the cumulative distribution function of the random variable under scrutiny. Depending on the probability distribution adopted, integrating this function may constitute a challenging task. For instance, if the random variable is normally distributed, no closed formulation exists for its cumulative distribution function. 
Several approximations have been proposed in the literature, see e.g. \cite{zelen64,shore1982,citeulike:12461167,citeulike:12461170,citeulike:12461174,citeulike:12317680}, which can be employed to approximate the first order loss function. However, these approximations are generally nonlinear and cannot be easily embedded in mixed integer linear programming (MILP) models. 

In Section \ref{sec:lb} and \ref{sec:ub}, we introduce piecewise linear lower and upper bounds for the first order loss function and its complementary function for the case of normally distributed random variables. These bounds are based on standard bounding techniques from stochastic programming, i.e. Jensen's lower bound  and Edmundson-Madansky upper bound \cite{citeulike:695971}, p. 167-168. The bounds can be readily used in MILP models and do not require instance dependent tabulations. Our linearisation strategy is based on standard optimal linearisation coefficients computed in such a way as to minimise the maximum approximation error, i.e. according to a minimax approach. Optimal coefficients for approximations comprising from two to eleven segments will be presented in Table \ref{tab:parameters}; these can be reused to approximate the loss function associated with any normally distributed random variable. 

\section{The first order loss function and its complementary function}\label{sec:loss_function}

Consider a continuous random variable $\omega$ with support over $\mathbb{R}$, probability density function $g_\omega(x):\mathbb{R}\rightarrow(0,1)$ and cumulative distribution function $G_\omega(x):\mathbb{R}\rightarrow(0,1)$. The first order loss function can be rewritten as
\begin{equation}
\mathcal{L}(x,\omega)=\int_{-\infty}^{\infty} \max(t-x,0)g_\omega(t)\,dt=\int_{x}^{\infty} (t-x)g_\omega(t)\,dt.
\end{equation}
The complementary first order loss function can be rewritten as
\begin{equation}
\widehat{\mathcal{L}}(x,\omega)=\int_{-\infty}^{\infty} \max(x-t,0)g_\omega(t)\,dt=\int_{-\infty}^{x} (x-t)g_\omega(t)\,dt.
\end{equation}

\begin{lem}\label{thm:relationship_fol_com_2}
The first order loss function $\mathcal{L}(x,\omega)$ can also be expressed as 
\begin{equation}
\mathcal{L}(x,\omega)=\int_{x}^{\infty} \left(1-G_\omega(t)\right)\,dt
\end{equation}
\end{lem}
\begin{proof}
\begin{align}
\mathcal{L}(x,\omega)		&=	\int_{x}^{\infty} (t-x)g_\omega(t)\,dt\\
						&=	\int_{x}^{\infty} t g_\omega(t)\,dt-x \int_{x}^{\infty} g_\omega(t)\,dt
\end{align}
the integration of 
\[\int_{x}^{\infty} t g_\omega(t)\,dt\]
is a well-known integration by parts that proceeds as follows. 
Let $u(t)=t$, $u'(t)=1$, $v(t)=-(1-G_\omega(t))$, $v'(t)=g_\omega(t)$. Rewrite the integral as
\[\int_{x}^{b} t g_\omega(t)\,dt=\left[-t(1-G_\omega(t))\right]_{x}^{b}+\int_{x}^{b}(1-G_\omega(t))\, dt\]
and take the limit since $b\rightarrow\infty$. The product term in the integration by parts formula converges to $x(1-G_\omega(x))$ as $b\rightarrow\infty$. We therefore take the limit to obtain the identity
\[\int_{x}^{\infty} t g_\omega(t)\,dt=x(1-G_\omega(x))+\int_{x}^{\infty}(1-G_\omega(t))\, dt\]
and by substituting $\int_{x}^{\infty} t g_\omega(t)\,dt$ with this expression we obtain
\begin{align}
\mathcal{L}(x,\omega)				&=	x(1-G_\omega(x))+\int_{x}^{\infty}(1-G_\omega(t))\, dt - x (1-G_\omega(x))\\
								&=	\int_{x}^{\infty}(1-G_\omega(t))\, dt 
\end{align}
\end{proof}

The following well-known lemma is introduced, together with its proof, for completeness.

\begin{lem}\label{thm:compl_fol}
The complementary first order loss function $\widehat{\mathcal{L}}(x,\omega)$ can also be expressed as
\begin{equation}
\widehat{\mathcal{L}}(x,\omega)=\int_{-\infty}^{x} G_\omega(t)\,dt.
\end{equation}
\end{lem}
\begin{proof}
\begin{align}
\widehat{\mathcal{L}}(x,\omega)		&=	\int_{-\infty}^{x} (x-t)g_\omega(t)\,dt\\
							&=	x \int_{-\infty}^{x} g_\omega(t)\,dt-\int_{-\infty}^{x} t g_\omega(t)\,dt\\											&=	x G_\omega(x) - \int_{-\infty}^{x} t g_\omega(t)\,dt\\	
							&=	x G_\omega(x) - x G_\omega(x) + \int_{-\infty}^{x} G_\omega(t)\,dt\\
							&=	\int_{-\infty}^{x} G_\omega(t)\,dt						
\end{align}
\end{proof}

There is a close relationship between the first order loss function and the complementary first order loss function.

\begin{lem}\label{thm:relationship_fol_com_1}
The first order loss function $\mathcal{L}(x,\omega)$ can also be expressed as 
\begin{equation}
\mathcal{L}(x,\omega)=\widehat{\mathcal{L}}(x,\omega)-(x-\tilde{\omega})
\end{equation}
where $\tilde{\omega}=\mbox{\em E}[\omega]$.
\end{lem}
\begin{proof}
\begin{align}
\mathcal{L}(x,\omega)		&=	\int_{x}^{\infty} (t-x)g_\omega(t)\,dt\\
						&=	\int_{x}^{\infty} t g_\omega(t)\,dt-x \int_{x}^{\infty} g_\omega(t)\,dt\\											&=	\int_{-\infty}^{\infty} t g_\omega(t)\,dt-\int_{-\infty}^{x} t g_\omega(t)\,dt-x \int_{x}^{\infty} g_\omega(t)\,dt\\
						&= 	\int_{-\infty}^{\infty} t g_\omega(t)\,dt-\int_{-\infty}^{x} t g_\omega(t)\,dt-x (1-G_\omega(t))\\
						&= 	\int_{-\infty}^{\infty} t g_\omega(t)\,dt- x G_\omega(x) + \int_{-\infty}^{x} G_\omega(t)\,dt-x (1-G_\omega(t))\\
						&=	\int_{-\infty}^{x} G_\omega(t)\,dt-(x-\tilde{\omega})\\
						&=	\widehat{\mathcal{L}}(x,\omega)-(x-\tilde{\omega})					
\end{align}
\end{proof}

Because of the relation discussed in Lemma \ref{thm:relationship_fol_com_1}, in what follows without loss of generality most of the results will be presented for the complementary first order loss function.

Another known result for the first order loss function and its complementary function is their convexity, which we present next.

\begin{lem}\label{lem:convexity_loss_function}
$\mathcal{L}(x,\omega)$ and $\widehat{\mathcal{L}}(x,\omega)$ are convex in $x$.
\end{lem}
\begin{proof}
We shall prove the result for $\widehat{\mathcal{L}}(x,\omega)$. Recall that $\widehat{\mathcal{L}}(x,\omega)=\int_{-\infty}^{x} G_\omega(t)\,dt$. From the fundamental theorem of integral calculus
\[\frac{d}{dx}\widehat{\mathcal{L}}(x,\omega)=G_\omega(x)\]
and
\[\frac{d^2}{dx^2}\widehat{\mathcal{L}}(x,\omega)=g_\omega(x).\]
Since $g_\omega(x)$ is nonnegative the result follows immediately; furthermore, the proof for $\mathcal{L}(x,\omega)$ follows from Lemma \ref{thm:relationship_fol_com_1} and from the fact that $-x$ is convex.
\end{proof}

For a random variable $\omega$ with symmetric probability density function, we introduce the following results.

\begin{lem}\label{thm:fol_symm}
If the probability density function of $\omega$ is symmetric about a mean value $\tilde{\omega}$, then \[\mathcal{L}(x,\omega)=\widehat{\mathcal{L}}(2\tilde{\omega}-x,\omega).\]
\end{lem}
\begin{proof}
\begin{align}
\mathcal{L}(x,\omega)		&=	\int_{x}^{\infty} \left(1-G_\omega(t)\right)\,dt	\\
						&=	\int_{-\infty}^{\tilde{\omega}-(x-\tilde{\omega})} G_\omega(t)\,dt	\\
						&=	\widehat{\mathcal{L}}(2\tilde{\omega}-x,\omega)		
\end{align}
\end{proof}

\begin{lem}\label{thm:fol_automorphism}
If the probability density function of $\omega$ is symmetric about a mean value $\tilde{\omega}$, then
\[\widehat{\mathcal{L}}(x,\omega)=\widehat{\mathcal{L}}(2\tilde{\omega}-x,\omega)+(x-\tilde{\omega})\] and
\[\mathcal{L}(x,\omega)=\mathcal{L}(2\tilde{\omega}-x,\omega)-(x-\tilde{\omega}).\]
\end{lem}
\begin{proof}
Follows immediately from Lemma \ref{thm:relationship_fol_com_1} and Lemma \ref{thm:fol_symm}.
\end{proof}

The results presented so far are easily extended to the case in which the random variable is discrete. In the following section we present results for the case in which the random variable is normally distributed.

\section{The first order loss function for a normally distributed random variable}\label{sec:loss_function_normal}

Let $\zeta$ be a normally distributed random variable with mean $\mu$ and standard deviation $\sigma$. 
Recall that the Normal probability density function is defined as
\begin{equation}
g_\zeta(x)=\frac{1}{\sigma\sqrt{2\pi}}e^{-\frac{(x-\mu)^2}{2\sigma^2}}.
\end{equation}
No closed form expression exists for the cumulative distribution function \[G_\zeta(x)=\int_{-\infty}^x g_\zeta(x)\,dx.\]
Let $\phi(x)$ be the standard Normal probability density function
and $\Phi(x)$ the respective cumulative distribution function.

We next present three known results for a normally distributed random variable: a standardisation result  in Lemma \ref{lem:folf_std_norm}, and two closed form expressions for the computation of the loss function and of its complementary function in Lemmas \ref{thm:compl_fol_closed} and \ref{thm:compl_fol_closed_2}.
\begin{lem}\label{lem:folf_std_norm}
The complementary first order loss function of $\zeta$ can be expressed in terms of the standard Normal cumulative distribution function as
\begin{equation}\label{eq:compl_fol_norm}
\widehat{\mathcal{L}}(x,\zeta)=\sigma\int_{-\infty}^{\frac{x-\mu}{\sigma}} \Phi(t)\,dt=\sigma\widehat{\mathcal{L}}\left(\frac{x-\mu}{\sigma},Z\right),
\end{equation}
where $Z$ is a standard Normal random variable.
\end{lem}
\begin{proof}
Recall that the complementary first order loss function is defined as
\[\widehat{\mathcal{L}}(x,\zeta)=\int_{-\infty}^{x} (x-t)g_\zeta(t)\,dt.\]
We change the upper integration limit to 
\[f(x)=\frac{x-\mu}{\sigma}\]
by noting that 
\[f^{-1}(y)=\sigma y+\mu,~~~~~\frac{df^{-1}(y)}{dy}=\sigma\]
it follows that 
\begin{align}
\widehat{\mathcal{L}}(x,\zeta)	&=\int_{-\infty}^{f(x)} (x-f^{-1}(t))g_\zeta(f^{-1}(t))\frac{df^{-1}(t)}{dt}\,dt.\\
						&=\sigma\int_{-\infty}^{\frac{x-\mu}{\sigma}} (x-\sigma t -\mu)\frac{1}{\sigma}\phi(t)\,dt.\\
						&=\sigma\int_{-\infty}^{\frac{x-\mu}{\sigma}} \sigma\left(\frac{x-\mu}{\sigma} - t\right)\frac{1}{\sigma}\phi(t)\,dt.
\end{align}
\end{proof}

\begin{lem}\label{thm:compl_fol_closed}
The complementary first order loss function $\widehat{\mathcal{L}}(x,\zeta)$ can be rewritten in closed form as
\[\widehat{\mathcal{L}}(x,\zeta)=\sigma\left(\phi\left(\frac{x-\mu}{\sigma}\right)+\Phi\left(\frac{x-\mu}{\sigma}\right)\frac{x-\mu}{\sigma}\right)\]
\end{lem}
\begin{proof}
Integrate by parts Eq. \ref{eq:compl_fol_norm} and observe that $\int_{-\infty}^{\frac{x-\mu}{\sigma}}t\phi(t)\,dt=-\phi(\frac{x-\mu}{\sigma})$.
\end{proof}

From Lemma \ref{thm:relationship_fol_com_1} the first order loss function of $\zeta$ can be expressed as
\begin{equation}\label{eq:fol_norm}
\mathcal{L}(x,\zeta)=-(x-\mu)+\sigma\int_{-\infty}^{\frac{x-\mu}{\sigma}} \Phi(t)\,dt
\end{equation}
Recall that an alternative expression is obtained via Lemma \ref{thm:relationship_fol_com_2},
\begin{equation}\label{eq:fol_norm_1}
\mathcal{L}(x,\zeta)=\sigma\int_{\frac{x-\mu}{\sigma}}^{\infty} (1-\Phi(t))\,dt.
\end{equation}
From Lemma \ref{thm:relationship_fol_com_1} and Lemma \ref{thm:compl_fol_closed} we obtain the closed form expression
\begin{equation}\label{eq:fol_norm_closed}
\mathcal{L}(x,\zeta)=-(x-\mu)+\sigma\left(\phi\left(\frac{x-\mu}{\sigma}\right)+\Phi\left(\frac{x-\mu}{\sigma}\right)\frac{x-\mu}{\sigma}\right)
\end{equation}

\begin{lem}\label{thm:compl_fol_closed_2}
The first order loss function $\mathcal{L}(x,\zeta)$ can be rewritten in closed form as
\[\mathcal{L}(x,\zeta)=\sigma\left(\phi\left(\frac{x-\mu}{\sigma}\right)-\left(1-\Phi\left(\frac{x-\mu}{\sigma}\right)\right)\frac{x-\mu}{\sigma}\right)\]
\end{lem}
\begin{proof}
Integrate by parts Eq. \ref{eq:fol_norm} and observe that $\int_{\frac{x-\mu}{\sigma}}^{\infty}t\phi(t)\,dt=\phi(\frac{x-\mu}{\sigma})$.
\end{proof}

Furthermore, since the Normal distribution is symmetric, both Lemma \ref{thm:fol_symm} and Lemma \ref{thm:fol_automorphism} hold.

\section{Jensen's lower bound for the standard normal first order loss function}\label{sec:lb}

We introduce a well-known inequality from stochastic programming \cite{citeulike:695971}, p. 167.

\subsection{Jensen's lower bound}

\begin{thm}[Jensen's inequality]\label{jensens}
Consider a random variable $\omega$ with support $\Omega$ and a function $f(x,s)$, which for a fixed $x$ is convex for all $s\in\Omega$, then
\[\mbox{E}[f(x,\omega)]\geq f(x,\mbox{E}[\omega]).\]
\end{thm}
\begin{proof}
\cite{citeulike:2516312}, p. 140.
\end{proof}

Common discrete lower bounding approximations in stochastic programming are extensions of Jensen's inequality. The usual strategy is to find a low cardinality discrete set of realisations representing a good approximation of the true underling distribution. 
 \cite{citeulike:2516312}, p. 288, discuss one of these discrete lower bounding approximations which consists in partitioning the support $\Omega$ into a number of disjoint regions, Jensen's bound is then applied in each of these regions.

More formally, let $g_{\omega}(\cdot)$ denote the probability density function of $\omega$ and consider a partition of the support $\Omega$ of $\omega$ into $N$ disjoint compact subregions $\Omega_1,\ldots,\Omega_N$. We define, for all $i=1,\ldots,N$ 
\[p_i=\Pr\{\omega\in \Omega_i\}=\int_{\Omega_i} g_{\omega}(t)\,dt\] 
and 
\[\mbox{E}[\omega|\Omega_i]=\frac{1}{p_i}\int_{\Omega_i} t g_{\omega}(t)\,dt\]

\begin{thm}\label{discrete_bounding}
\[\mbox{E}[f(x,\omega)]\geq\sum_{i=1}^N p_i f(x,\mbox{E}[\omega|\Omega_i])\]
\end{thm}
\begin{proof}
\cite{citeulike:2516312}, p. 289.
\end{proof}

\begin{thm}
Given a random variable $\omega$ Jensen's bound (Theorem \ref{jensens}) is applicable to the first order loss function $\mathcal{L}(x,\omega)$ and its complementary function $\widehat{\mathcal{L}}(x,\omega)$. 
\end{thm}
\begin{proof}
Follows immediately from Lemma \ref{lem:convexity_loss_function}.
\end{proof}
Having established this result, we must then decide how to partition the support $\omega$ in order to obtain a good bound. In fact, to generate good lower bounds, it is necessary to carefully select the partition of the support $\omega$. The optimal partitioning strategy will depend, of course, on the probability distribution of the random variable $\omega$. 

\subsection{Minimax discrete lower bounding approximation}

We discuss a minimax strategy for generating discrete lower bounding approximations of the (complementary) first order loss function. In this strategy, we partition the support $\omega$ into a predefined number of regions $N$ in order to minimise the maximum approximation error.

Consider a random variable $\omega$ and the associated complementary first order loss function
\[\widehat{\mathcal{L}}(x,\omega)=\mbox{E}[\max(x-\omega,0)];\]
assume that the support $\Omega$ of $\omega$ is partitioned into $N$ disjoint subregions $\Omega_1,\ldots,\Omega_N$.

\begin{lem}\label{lem:piecewise_linear}
For the (complementary) first order loss function the lower bound presented in Theorem \ref{discrete_bounding} is a piecewise linear function with $N+1$ segments.
\end{lem}
\begin{proof}
Consider the bound presented in Theorem \ref{discrete_bounding} and let $f(x,\omega)=\max(x-\omega,0)$, 
\[\widehat{\mathcal{L}}_{lb}(x,\omega)=\sum_{i=1}^N p_i \max(x-\mbox{E}[\omega|\Omega_i],0)\]
this function is equivalent to
{\scriptsize
\[\widehat{\mathcal{L}}_{lb}(x,\omega)=\left\{
\begin{array}{ll}
0&-\infty\leq x\leq \mbox{E}[\omega|\Omega_1]\\
p_1 x - p_1\mbox{E}[\omega|\Omega_1]&\mbox{E}[\omega|\Omega_1]\leq x\leq \mbox{E}[\omega|\Omega_2]\\
(p_1+p_2) x - (p_1\mbox{E}[\omega|\Omega_1]+p_2\mbox{E}[\omega|\Omega_2])&\mbox{E}[\omega|\Omega_2]\leq x\leq \mbox{E}[\omega|\Omega_3]\\
\vdots&\vdots\\
(p_1+p_2+\hdots+p_N) x - (p_1\mbox{E}[\omega|\Omega_1]+p_2\mbox{E}[\omega|\Omega_2]+\hdots+p_N\mbox{E}[\omega|\Omega_N])&\mbox{E}[\omega|\Omega_{N-1}]\leq x\leq \mbox{E}[\omega|\Omega_N]\\
\end{array}
\right.\]\\}
which is piecewise linear in $x$ with breakpoints at $\mbox{E}[\omega|\Omega_1],\mbox{E}[\omega|\Omega_2],\ldots,\mbox{E}[\omega|\Omega_N]$. The proof for the first order loss function follows a similar reasoning.
\end{proof}

\begin{lem}\label{lem:tangent}
Consider the $i$-th linear segment of $\widehat{\mathcal{L}}_{lb}(x,\omega)$
\[\widehat{\mathcal{L}}^i_{lb}(x,\omega)=x\sum_{k=1}^i p_k-\sum_{k=1}^i p_k \mbox{E}[\omega|\Omega_k]~~~~\mbox{E}[\omega|\Omega_{i}]\leq x\leq \mbox{E}[\omega|\Omega_{i+1}],\]
where $i=1,\ldots,N$. Let $\Omega_i=[a,b]$, then $\widehat{\mathcal{L}}^i_{lb}(x,\omega)$ is tangent to $\widehat{\mathcal{L}}(x,\omega)$ at $x=b$. Furthermore, the $0$-th segment $x=0$ is tangent to $\widehat{\mathcal{L}}(x,\omega)$ at $x=-\infty$.
\end{lem}
\begin{proof}
Note that
\[\widehat{\mathcal{L}}^i_{lb}(x,\omega)=x\sum_{k=1}^i \int_{\Omega_k} g_{\omega}(t)\;dt-\sum_{k=1}^i \int_{\Omega_k} t g_{\omega}(t)\;dt\]
and that \[\Omega_1\cup\Omega_2\cup\ldots\cup\Omega_i=)-\infty,b]\]
it follows
\[\widehat{\mathcal{L}}^i_{lb}(x,\omega)=x \int_{-\infty}^b g_{\omega}(t)\;dt-\int_{-\infty}^b t g_{\omega}(t)\;dt\]
and
\[\widehat{\mathcal{L}}^i_{lb}(x,\omega)=G_{\omega}(b)(x-b)+\int_{-\infty}^b G_{\omega}(t)\;dt.\]
which is the equation of the tangent line to $\widehat{\mathcal{L}}(x,\omega)$ at a given point $b$, that is
\[y=\widehat{\mathcal{L}}(b,\omega)'(x-b)+\widehat{\mathcal{L}}(b,\omega).\]
$x=0$ is tangent to $\widehat{\mathcal{L}}(x,\omega)$ at $x=-\infty$ since $\widehat{\mathcal{L}}(x,\omega)$ is convex, positive and
\[\lim_{x\rightarrow -\infty} \widehat{\mathcal{L}}(x,\omega)=0.\]
The very same reasoning can be easily applied to the first order loss function.
\end{proof}

\begin{lem}\label{lem:max_error_at_breakpoints}
The maximum approximation error between $\widehat{\mathcal{L}}_{lb}(x,\omega)$ and $\widehat{\mathcal{L}}(x,\omega)$ will be attained at a breakpoint.
\end{lem}
\begin{proof}
By recalling that $\widehat{\mathcal{L}}(x,\omega)$ is convex (Lemma \ref{lem:convexity_loss_function}), since $\widehat{\mathcal{L}}_{lb}(x,\omega)$ is piecewise linear (Lemma \ref{lem:piecewise_linear}) and each segment of $\widehat{\mathcal{L}}_{lb}(x,\omega)$ is tangent to $\widehat{\mathcal{L}}(x,\omega)$ (Lemma \ref{lem:tangent}), it follows that the maximum error will be attained at a breakpoint. 
\end{proof}

\begin{thm}\label{thm:approx_err_equal}
Given the number of regions $N$, $\Omega_1,\ldots,\Omega_N$ is an optimal partition of the support $\Omega$ of $\omega$ under a minimax strategy, if and only if approximation errors at breakpoints are all equal.
\end{thm}
\begin{proof}
The approximation errors for $x\rightarrow-\infty$ and $x\rightarrow\infty$ are both 0; since we have N+1 segments, we only have $N$ breakpoints to check. 

We first show that ($\rightarrow$) if $\Omega_1,\ldots,\Omega_N$ is an optimal partition of the support $\Omega$ of $\omega$ under a minimax strategy, then approximation errors at breakpoints are all equal. 

A first observation that follows immediately from Lemma \ref{lem:max_error_at_breakpoints} is that, if the slope of segment $\widehat{\mathcal{L}}^{i+1}_{lb}(x,\omega)$ remains unchanged, and the breakpoint between $\widehat{\mathcal{L}}^i_{lb}(x,\omega)$ and $\widehat{\mathcal{L}}^{i+1}_{lb}(x,\omega)$ moves towards the point at which $\widehat{\mathcal{L}}^{i+1}_{lb}(x,\omega)$ is tangent to $\widehat{\mathcal{L}}(x,\omega)$, the error at such breakpoint decreases.


If one changes the size of region $\Omega_i$ so that the upper limit becomes $b_i+\Delta$, then $\mbox{E}[\omega|\Omega_i]$ will increase if $\Delta>0$, or will decrease if $\Delta<0$. This immediately follows from the definition of $\mbox{E}[\omega|\Omega_i]$. Therefore, the breakpoint between segment $i$ and segment $i+1$, which occurs at $\mbox{E}[\omega|\Omega_i]$, will move accordingly. However, the slope of segment $i+1$, which we recall is equal to $G_\omega(b_{i+1})$, depends uniquely on the upper limit of the region $\Omega_{i+1}$, $b_{i+1}$, and is not affected by a change in the upper limit of region $\Omega_i$. Therefore, the error at the breakpoint between segment $i$ and segment $i+1$ will decrease if $\Delta>0$, or will increase if $\Delta<0$.

Now, assume that $\Omega_1,\ldots,\Omega_N$ is an optimal partition of the support $\Omega$ of $\omega$ and approximation errors at breakpoints are not all equal. Furthermore, assume that the maximum approximation error occurs at breakpoint $i$. By  increasing the size of the region $\Omega_i$, i.e. by setting the upper limit to $b_i+\Delta$, where $\Delta>0$, it is possible to decrease the maximum error until it becomes equal to the error at breakpoint $k$, where $k\in\{1,\ldots,i-1\}$. The procedure can be repeated until all approximation errors are equal.

Second, we show that ($\leftarrow$) if approximation errors at breakpoints are all equal, then $\Omega_1,\ldots,\Omega_N$ is an optimal partition of the support $\Omega$ of $\omega$ under a minimax strategy. 

If approximation errors at breakpoints are all equal and we change the size of region $\Omega_i$ by setting the upper limit to $b_i+\Delta$, where $\Delta>0$, then the approximation error at breakpoint $i-1$ will increase; conversely, if $\Delta<0$, then the approximation error at breakpoint $i+1$ will increase.
\end{proof}

By using this last result it is possible to derive a set of equations that can be solved for computing an optimal partitioning.
Let us consider the error $e_i$ at breakpoint $i$, this can be expressed as 
\[e_i=\widehat{\mathcal{L}}(\mbox{E}[\omega|\Omega_i],\omega)-\widehat{\mathcal{L}}^{i}_{lb}(\mbox{E}[\omega|\Omega_i],\omega),\]
where 
$\Omega_i=[a_i,b_i]$.
Since we have $N$ breakpoints to check, we must solve a system comprising the following $N-1$ equations 
\[e_1=e_i~~~\mbox{for }i=2,\ldots,N.\]
under the following restrictions
\[
\begin{array}{lll}
a_1&=-\infty\\
b_N&=\infty\\
a_i&\leq b_i&\mbox{for }i=1,\ldots,N\\
b_i&=a_{i+1}&\mbox{for }i=1,\ldots,N-1
\end{array}
\]
The system therefore involves $N-1$ variables, each of which identifies the boundary between two disjoint regions $\Omega_i$ and $\Omega_{i+1}$. 

\begin{thm}\label{thm:approx_err_symmetric}
Assume that the probability density function of $\omega$ is symmetric about a mean value $\tilde{\omega}$. Then, under a minimax strategy, if $\Omega_1,\ldots,\Omega_N$ is an optimal partition of the support $\Omega$ of $\omega$, breakpoints will be symmetric about $\tilde{\omega}$.
\end{thm}
\begin{proof}
This follows from Lemma \ref{thm:fol_automorphism} and Theorem \ref{thm:approx_err_equal}.
\end{proof}

In this case, by exploiting the symmetry of the piecewise linear approximation, an optimal partitioning can be derived by solving a smaller system comprising $\lceil N/2 \rceil$ equations, where $N$ is the number of regions $\Omega_i$ and $\lceil x \rceil$ rounds $x$ to the next integer value.

Unfortunately, equations in the above system are nonlinear and do not admit a closed form solution in the general case. 

\subsubsection{Normal distribution}

We will next discuss the system of equations that leads to an optimal partitioning for the case of a standard Normal random variable $Z$. This partitioning leads to a piecewise linear approximation that is, in fact, easily extended to the general case of a normally distributed variable $\zeta$ with mean $\mu$ and standard deviation $\sigma$ via Lemma \ref{lem:folf_std_norm}. This equation suggests that the error of this approximation is independent of $\mu$ and proportional to $\sigma$.

Consider a partitioning for the support $\Omega$ of $Z$ into $N$ adjacent regions $\Omega_i=[a_i,b_i]$, where $i=1,\ldots,N$.
From Theorem \ref{thm:approx_err_symmetric}, if $N$ is odd, then $b_{\lceil N/2\rceil}=0$ and $b_{i}=-b_{N+1-i}$, if $N$ is even, then $b_{i}=-b_{N+1-i}$. We shall use Lemma \ref{thm:compl_fol_closed} for expressing $\widehat{\mathcal{L}}(x,Z)$.
Then, by observing that
\[\int_{a_i}^{b_i} t \phi(t)\,dt=\phi(a_i)-\phi(b_i)\]
and that $p_1+p_2+\ldots+p_i=\Phi(b_i)$, we rewrite
\begin{align}
\widehat{\mathcal{L}}^i_{lb}(\mbox{E}[Z|\Omega_i],Z)	&=\Phi(b_i)\mbox{E}[Z|\Omega_i]-\sum_{k=1}^i(\phi(a_i)-\phi(b_i))\\
							&=\Phi(b_i)\mbox{E}[Z|\Omega_i]-(\phi(a_1)-\phi(b_i))\\
							&=\Phi(b_i)\mbox{E}[Z|\Omega_i]+\phi(b_i)
\end{align}

To express the conditional expectation $\mbox{E}[Z|\Omega_i]$
we proceed as follows: let $p_i=\Phi(b_i)-\Phi(a_i)$,
it follows
\[\mbox{E}[Z|\Omega_i]=\frac{\phi(a_i)-\phi(b_i)}{\Phi(b_i)-\Phi(a_i)}.\]

To solve the above system of non-linear equations we will exploit the close connections between finding a local minimum and solving a set of nonlinear equations. In particular, we will use the Gauss-Newton method to find a partition $\Omega_1,\ldots,\Omega_N$ of the support of $Z$ that minimises the following sum of squares
\[\sum_{k=2}^{N}(e_1-e_k)^2\]

This minimisation problem can be solved by software packages such as Mathematica (see \texttt{NMinimize}). 

\subsubsection{Numerical examples}

The classical Jensen's bound for the complementary first order loss function of a standard Normal random variable $Z$ is show in Fig. \ref{fig:piecewise-2}. This is obtained by considering a degenerate partition of the support of $Z$ comprising only a single region $\Omega_1=[-\infty,\infty]$. In practice, we simply replace $Z$ by its expected value, i.e. zero. Therefore we simply have $\widehat{\mathcal{L}}_{lb}=\max(x,0)$. The maximum error of this piecewise linear approximation occurs for $x=0$ and it is equal to $1/\sqrt{2\pi}$.
\begin{figure}[h!]
\centering
\includegraphics[type=eps,ext=.eps,read=.eps,width=0.8\columnwidth]{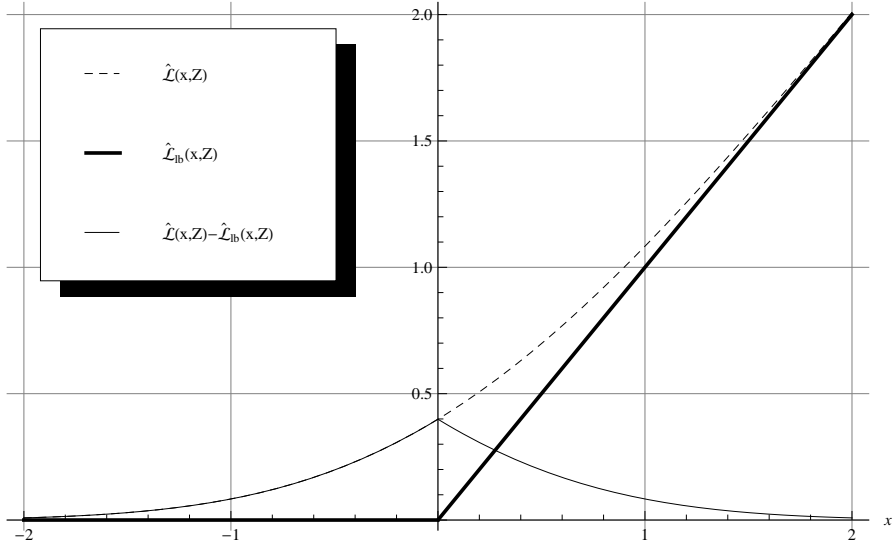}
\caption{classical Jensen's bound for $\widehat{\mathcal{L}}(x,Z)$}
\label{fig:piecewise-2}
\end{figure}   



If we split the support of $Z$ into four regions (Fig. \ref{fig:piecewise-5}), the solution to the system of nonlinear equations prescribes to split $\Omega$ at $b_1=-0.886942$, $b_2=0$, $b_3=0.886942$. The maximum error is 0.0339052 and it is observed at $x\in\{\pm1.43535, \pm0.415223\}$.
\begin{figure}[h!]
\centering
\includegraphics[type=eps,ext=.eps,read=.eps,width=0.8\columnwidth]{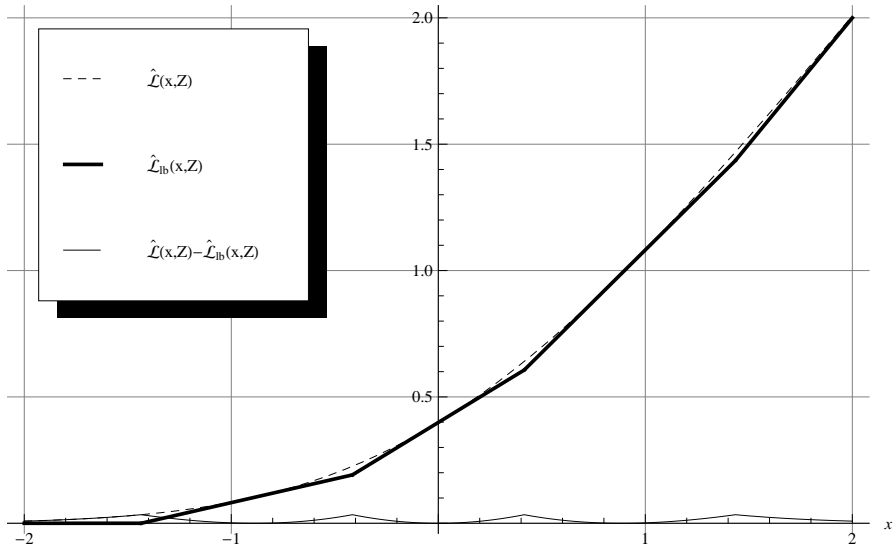}
\caption{five-segment piecewise Jensen's bound for $\widehat{\mathcal{L}}(x,Z)$}
\label{fig:piecewise-5}
\end{figure}   

In Table \ref{tab:parameters} we report parameters of $\widehat{\mathcal{L}}_{lb}(x,Z)$ with up to eleven segments. In Fig. \ref{fig:error} we present the approximation error of $\widehat{\mathcal{L}}_{lb}(x,Z)$ with up to eleven segments.
\begin{figure}[h!]
\centering
\includegraphics[type=eps,ext=.eps,read=.eps,width=0.8\columnwidth]{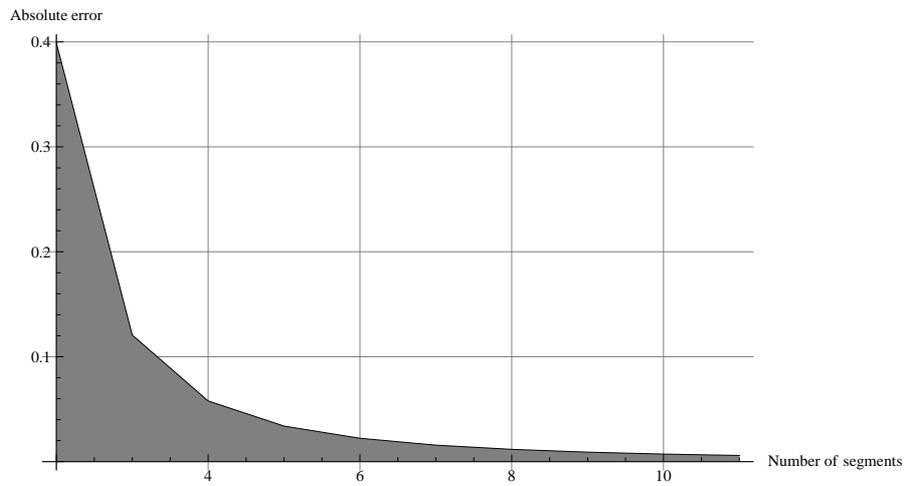}
\caption{approximation error of $\widehat{\mathcal{L}}_{lb}(x,Z)$ with up to eleven segments}
\label{fig:error}
\end{figure}   

\begin{landscape}
\begin{table}[htdp]
\tiny
\begin{center}
\begin{tabular}{l|l|lllllllllll}
&&\multicolumn{11}{c}{Piecewise linear approximation parameters}\\
Segments&Error&$i$&1&2&3&4&5&6&7&8&9&10\\
\hline
\multirow{4}{*}{2}&
\multirow{4}{*}{0.398942}
&$b_i$&$\infty$\\
&&$p_i$&1\\
&&$\mbox{E}[\omega|\Omega_i]$&0\\
\hline
\multirow{4}{*}{3}&
\multirow{4}{*}{0.120656}
&$b_i$&0&$\infty$\\
&&$p_i$&0.5&0.5\\
&&$\mbox{E}[\omega|\Omega_i]$&$-0.797885$&$0.797885$\\
\hline
\multirow{4}{*}{4}&
\multirow{4}{*}{0.0578441}
&$b_i$&$-0.559725$&$0.559725$&$\infty$\\
&&$p_i$&0.287833&0.424333&0.287833\\
&&$\mbox{E}[\omega|\Omega_i]$&$-1.18505$&$0$&$1.18505$\\
\hline
\multirow{4}{*}{5}&
\multirow{4}{*}{0.0339052}
&$b_i$&$-0.886942$&$0$&$0.886942$&$\infty$\\
&&$p_i$&$0.187555$&$0.312445$&$0.312445$&$0.187555$\\
&&$\mbox{E}[\omega|\Omega_i]$&$-1.43535$&$-0.415223$&$0.415223$&$1.43535$\\
\hline
\multirow{4}{*}{6}&
\multirow{4}{*}{0.0222709}
&$b_i$&$-1.11507$&$-0.33895$&$0.33895$&$1.11507$&$\infty$\\
&&$p_i$&$0.132411$&$0.234913$&$0.265353$&$0.234913$&$0.132411$\\
&&$\mbox{E}[\omega|\Omega_i]$&$-1.61805$&$-0.691424$&$0$&$0.691424$&$1.61805$\\
\hline
\multirow{4}{*}{7}&
\multirow{4}{*}{0.0157461}
&$b_i$&$-1.28855$&$-0.579834$&$0$&$0.579834$&$1.28855$&$\infty$\\
&&$p_i$&$0.0987769$&$0.182236$&$0.218987$&$0.218987$&$0.182236$&$0.0987769$\\
&&$\mbox{E}[\omega|\Omega_i]$&$-1.7608$&$-0.896011$&$-0.281889$&$0.281889$&$0.896011$&$1.7608$\\
\hline
\multirow{4}{*}{8}&
\multirow{4}{*}{0.0117218}
&$b_i$&$-1.42763$&$-0.765185$&$-0.244223$&$0.244223$&$0.765185$&$1.42763$&$\infty$\\
&&$p_i$&$0.0766989$&$0.145382$&$0.181448$&$0.192942$&$0.181448$&$0.145382$&$0.0766989$\\
&&$\mbox{E}[\omega|\Omega_i]$&$-1.87735$&$-1.05723$&$-0.493405$&$0$&$0.493405$&$1.05723$&$1.87735$\\
\hline
\multirow{4}{*}{9}&
\multirow{4}{*}{0.00906529}
&$b_i$&$-1.54317$&$-0.914924$&$-0.433939$&$0$&$0.433939$&$0.914924$&$1.54317$&$\infty$\\
&&$p_i$&$0.0613946$&$0.118721$&$0.152051$&$0.167834$&$0.167834$&$0.152051$&$0.118721$&$0.0613946$\\
&&$\mbox{E}[\omega|\Omega_i]$&$-1.97547$&$-1.18953$&$-0.661552$&$-0.213587$&$0.213587$&$0.661552$&$1.18953$&$1.97547$\\
\hline
\multirow{4}{*}{10}&
\multirow{4}{*}{0.00721992}
&$b_i$&$-1.64166$&$-1.03998$&$-0.58826$&$-0.19112$&$0.19112$&$0.58826$&$1.03998$&$1.64166$&$\infty$\\
&&$p_i$&$0.0503306$&$0.0988444$&$0.129004$&$0.146037$&$0.151568$&$0.146037$&$0.129004$&$0.0988444$&$0.0503306$\\
&&$\mbox{E}[\omega|\Omega_i]$&$-2.05996$&$-1.30127$&$-0.8004$&$-0.384597$&$0.$&$0.384597$&$0.8004$&$1.30127$&$2.05996$\\
\hline
\multirow{4}{*}{11}&
\multirow{4}{*}{0.00588597}
&$b_i$&$-1.72725$&$-1.14697$&$-0.717801$&$-0.347462$&$0.$&$0.347462$&$0.717801$&$1.14697$&$1.72725$&$\infty$\\
&&$p_i$&$0.0420611$&$0.0836356$&$0.110743$&$0.127682$&$0.135878$&$0.135878$&$0.127682$&$0.110743$&$0.0836356$&$0.0420611$\\
&&$\mbox{E}[\omega|\Omega_i]$&$-2.13399$&$-1.39768$&$-0.9182$&$-0.526575$&$-0.17199$&$0.17199$&$0.526575$&$0.9182$&$1.39768$&$2.13399$
\end{tabular}
\end{center}
\caption{parameters of $\widehat{\mathcal{L}}_{lb}(x,Z)$ with up to eleven segments}
\label{tab:parameters}
\end{table}%
\end{landscape}

In Fig. \ref{fig:piecewise-5-scaled} we exploited Lemma \ref{lem:folf_std_norm} to obtain
the five-segment piecewise Jensen's bound for $\widehat{\mathcal{L}}(x,\zeta)$, where $\zeta$ is a normally distributed random variable with mean $\mu=20$ and standard deviation $\sigma=5$. The maximum error is $\sigma0.0339052$ and it is observed at $x\in\{\sigma(\pm1.43535)+\mu, \sigma(\pm0.415223)+\mu\}$.
\begin{figure}[h!]
\centering
\includegraphics[type=eps,ext=.eps,read=.eps,width=0.8\columnwidth]{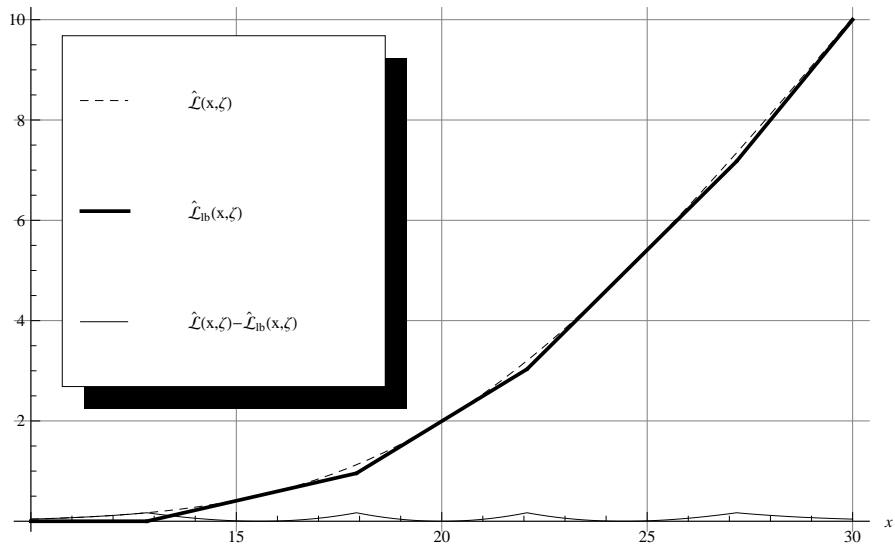}
\caption{five-segment piecewise Jensen's bound for $\widehat{\mathcal{L}}(x,\zeta)$, where $\mu=20$ and $\sigma=5$}
\label{fig:piecewise-5-scaled}
\end{figure}   

\section{An approximate piecewise linear upper bound for the standard normal first order loss function}\label{sec:ub}

In this section we introduce a simple bounding technique that exploits convexity of the (complementary) first order loss function to derive a piecewise linear upper bound. 

\subsection{A piecewise linear upper bound}

Without loss of generality we shall introduce the bound for the complementary first order loss function. Consider a random variable $\omega$ with support $\Omega$. From Lemma \ref{lem:convexity_loss_function}, $\widehat{\mathcal{L}}(x,\omega)$ is convex in $x$ regardless of the distribution of $\omega$. Given an interval $[a,b]\in\mathbb{R}$, it is possible to construct an upper bound by exploiting the very same definition of convexity, that is by constructing a straight line $\widehat{\mathcal{L}}_{ub}(x,\omega)$ between the two points $(a,\widehat{\mathcal{L}}(a,\omega))$ and $(b,\widehat{\mathcal{L}}(b,\omega))$. The slope ($\alpha$) and the intercept ($\beta$) of this line can be easily computed
\[\alpha=\frac{\widehat{\mathcal{L}}(b,\omega)-\widehat{\mathcal{L}}(a,\omega)}{b-a},~~~~~\beta=\frac{b \widehat{\mathcal{L}}(a,\omega)- a \widehat{\mathcal{L}}(b,\omega)}{b-a}.\]
The upper bound is then 
\begin{align}
\widehat{\mathcal{L}}_{ub}(x,\omega)&=\alpha x+\beta&a\leq x\leq b\\
							&=\widehat{\mathcal{L}}(a,\omega) \frac{b - x}{b-a}+\widehat{\mathcal{L}}(b,\omega) \frac{x - a}{b-a}&a\leq x\leq b
\end{align}
We can improve the quality of this bound by partitioning the domain $\mathbb{R}$ of $\widehat{\mathcal{L}}(x,\omega)$ into $N$ disjoint regions $\mathcal{D}_i=[a_i,b_i]$, $i=1,\ldots,N$. The selected regions must be all compact and adjacent. Because of the convexity of $\widehat{\mathcal{L}}(x,\omega)$ the bound can be then applied to each of these regions considered separately. 

However, since $\widehat{\mathcal{L}}(x,\omega)$ is defined over $\mathbb{R}$, it is not possible to guarantee a complete covering of the domain by using compact regions. We must therefore add two extreme regions $\mathcal{D}_0=[-\infty,a_1]$ and $\mathcal{D}_{N+1}=[a_{N+1}=b_N,\infty]$ to ensure the one obtained is indeed an upper bound for each $x\in\mathbb{R}$. By noting that
\[
\lim_{x\rightarrow_{-\infty}}\widehat{\mathcal{L}}(x,\omega)=0~~~\mbox{and}~~~\lim_{x\rightarrow_{\infty}}\widehat{\mathcal{L}}(x,\omega)=x
\]
it is easy to derive equations for the lines associated with these two extra regions. In particular, we associate with $\mathcal{D}_0$ a horizontal line with slope $\alpha=0$ and intercept $\beta=\widehat{\mathcal{L}}(a_1,\omega)$, and with $\mathcal{D}_{N+1}$ a line with slope $\alpha=1$ and intercept $\beta=\widehat{\mathcal{L}}(b_N,\omega)-b_N$.

Also in this case, we must then decide how to partition the domain $\mathbb{R}$ into $N+2$ intervals $\mathcal{D}_0,\ldots,\mathcal{D}_{N+1}$ to obtain a tight bound. Once more, the optimal partitioning strategy will depend on the probability distribution of the random variable $\omega$. 

\subsection{Minimax piecewise linear upper bound}

We discuss a minimax strategy for generating a piecewise linear upper bound of the (complementary) first order loss function $\widehat{\mathcal{L}}(x,\omega)$. In this strategy, we partition of the domain $\mathbb{R}$ of $x$ into a predefined number of regions $N+2$ in order to minimise the maximum approximation error. Note that, since this domain is not compact, one needs at least two regions to derive a piecewise linear upper bound.

Consider a random variable $\omega$ and the associated complementary first order loss function 
\[\widehat{\mathcal{L}}(x,\omega)=\mbox{E}[\max(x-\omega,0)];\]
assume that the domain $\mathbb{R}$ of $x$ is partitioned into $N+2$ disjoint adjacent subregions $\mathcal{D}_0,\ldots,\mathcal{D}_{N+1}$, where $\mathcal{D}_0=[-\infty,a_1]$, $\mathcal{D}_i=[a_i,b_i]$, for $i=1,\ldots,N$, and $\mathcal{D}_{N+1}=[b_N,\infty]$, and consider the following piecewise linear upper bound
\[
\widehat{\mathcal{L}}_{ub}(x,\omega)=\left\{
\begin{array}{ll}
\widehat{\mathcal{L}}(a_1,\omega)&x\in\mathcal{D}_0\\
\vdots\\
\widehat{\mathcal{L}}(a_i,\omega) \frac{b_i - x}{b_i-a_i}+\widehat{\mathcal{L}}(b_i,\omega) \frac{x - a_i}{b_i-a_i}&x\in\mathcal{D}_i\\
\vdots\\
x+\widehat{\mathcal{L}}(b_N,\omega)-b_N&x\in\mathcal{D}_{N+1}
\end{array}
\right.
\]
Let $\widehat{\mathcal{L}}^i_{ub}(x,\omega)$ be the linear segment of $\widehat{\mathcal{L}}_{ub}(x,\omega)$ over $\mathcal{D}_i$, for $i=0,\ldots,N+1$.
\begin{lem}\label{lem:max_error_ub}
Consider $\widehat{\mathcal{L}}^i_{ub}(x,\omega)$, where $i=1,\ldots,N$; the maximum approximation error between $\widehat{\mathcal{L}}^i_{ub}(x,\omega)$ and $\widehat{\mathcal{L}}(x,\omega)$ will be attained for \[\bar{x}_i=G^{-1}_\omega\left(\frac{\widehat{\mathcal{L}}(b_i,\omega)-\widehat{\mathcal{L}}(a_i,\omega)}{b_i-a_i}\right).\]
\end{lem}
\begin{proof}
The idea here is to derive a line that is tangent to $\widehat{\mathcal{L}}(x,\omega)$ and that has a slope equal to that of the $i$-th linear segment of $\widehat{\mathcal{L}}_{ub}(x,\omega)$. We have already discussed in Lemma \ref{lem:tangent} that the equation of the tangent to $\widehat{\mathcal{L}}(x,\omega)$ at a given point $\bar{x}_i$ is 
\[y=\widehat{\mathcal{L}}(\bar{x}_i,\omega)'(x-p)+\widehat{\mathcal{L}}(\bar{x}_i,\omega)\]
that is
\[\widehat{\mathcal{L}}^i_{lb}(x,\omega)=G_{\omega}(\bar{x}_i)(x-p)+\int_{-\infty}^{\bar{x}_i} G_{\omega}(t)\;dt.\]
The slope $G_{\omega}(\bar{x}_i)$ only depends on $\bar{x}_i$. To find a tangent with a slope equal to that of the $i$-th linear segment of $\widehat{\mathcal{L}}_{ub}(x,\omega)$, we simply let \[G_{\omega}(\bar{x}_i)=\frac{\widehat{\mathcal{L}}(b_i,\omega)-\widehat{\mathcal{L}}(a_i,\omega)}{b_i-a_i}\] and invert the cumulative distribution function.
\end{proof}
Note that the maximum approximation error for the linear segment over $\mathcal{D}_0$ is $\widehat{\mathcal{L}}(a_1,\omega)$ and the maximum approximation error for the linear segment over $\mathcal{D}_{N+1}$ is $\widehat{\mathcal{L}}(b_N,\omega)-b_N$. This can be inferred from the fact that $\widehat{\mathcal{L}}(x,\omega)$ monotonically approaches 0 for $x\rightarrow-\infty$ and $x$ for $x\rightarrow\infty$.
\begin{thm}\label{thm:approx_err_equal_ub}
$\mathcal{D}_0,\ldots,\mathcal{D}_{N+1}$ is an optimal partition under a minimax strategy, if and only if the maximum approximation error between $\widehat{\mathcal{L}}(x,\omega)$ and each linear segment of $\widehat{\mathcal{L}}_{ub}(x,\omega)$ is the same.
\end{thm}
\begin{proof}
The proof of this theorem can be obtained from Lemma \ref{lem:max_error_ub} and from Theorem \ref{thm:approx_err_equal}. In particular, the key insight needed to understand this result is the following. In Theorem \ref{thm:approx_err_equal} we showed that the approximation errors at breakpoints for the piecewise linear lower bound presented are all equal to each other and also equal to the maximum approximation error; furthermore, in Lemma \ref{lem:tangent} we showed that the $i$-th piecewise linear segment of this lower bound agrees with the original function at point $b_i$, where $\Omega_i=[a_i,b_i]$ is the $i$-th partition of the support of $Z$, for $i=1,\ldots,N$.  Since the first order loss function is convex and we know the maximum approximation error, by shifting up the piecewise linear lower bound by a value equal to the maximum approximation error we immediately obtain a piecewise linear upper bound comprising $N+1$ segments. This upper bound agrees with the original function at $\mbox{E}[Z|\Omega_i]$, for $i=1,\ldots,N$. The maximum approximation error will be attained at those points in which the lower bound was tangent to the original function, that is $a_1,b_1,b_2,\ldots,b_N$. By using a reasoning similar to the one developed for Theorem \ref{thm:approx_err_equal}, it is possible to show that, if we increase or decrease at least one $\mbox{E}[Z|\Omega_i]$, the maximum approximation error can only increase.
\end{proof}
By using this result it is possible to derive a set of equations that can be solved for computing an optimal partitioning.
Let us consider the maximum approximation error $e_i$ associated with the $i$-th linear segment of $\widehat{\mathcal{L}}_{ub}(x,\omega)$, this can be expressed as 
\[e_i=\widehat{\mathcal{L}}^{i}_{ub}(\bar{x}_i,\omega)-\widehat{\mathcal{L}}(\bar{x}_i,\omega),\]
where $i=1,\ldots,N$; furthermore $e_0=\widehat{\mathcal{L}}(a_1,\omega)$ and $e_{N+1}=\widehat{\mathcal{L}}(b_N,\omega)-b_N$.
Since we have $N+2$ segments to check, we must solve a system comprising the following $N+1$ equations 
\[e_0=e_i~~~\mbox{for }i=1,\ldots,N+1\]
under the following restrictions
\[
\begin{array}{lll}
a_i&\leq b_i&\mbox{for }i=1,\ldots,N\\
b_i&=a_{i+1}&\mbox{for }i=1,\ldots,N-1
\end{array}
\]
The system involves $N+1$ variables, each of which identifies the boundary between two disjoint adjacent regions $\mathcal{D}_i$ and $\mathcal{D}_{i+1}$. 
\begin{thm}\label{thm:approx_err_symmetric_ub}
Assume that the probability density function of $\omega$ is symmetric about a mean value $\tilde{\omega}$. Then, under a minimax strategy, if $\mathcal{D}_0,\ldots,\mathcal{D}_{N+1}$ is an optimal partition of the domain, breakpoints will be symmetric about the mean value $\tilde{\omega}$.
\end{thm}
\begin{proof}
This follows from Lemma \ref{thm:fol_automorphism} and Theorem \ref{thm:approx_err_equal_ub}.
\end{proof}
In this case, by exploiting the symmetry of the piecewise linear approximation, an optimal partitioning can be derived by solving a smaller system comprising $\lceil (N+1)/2 \rceil$ equations, where $N+2$ is the number of regions $\mathcal{D}_i$ and $\lceil x \rceil$ rounds $x$ to the next integer value.

As in the case of Jensen's bound, equations in the above system are nonlinear and do not admit a closed form solution in the general case. For sake of completeness, we will briefly discuss next how to derive the system of nonlinear equations for the case of a standard normal random variable $Z$. However, one should note that in practice, by exploiting the properties illustrated in the proof of Theorem \ref{thm:approx_err_equal_ub}, one does not need to solve a new system of nonlinear equations to derive the piecewise linear upper bound. All information needed, i.e. maximum approximation error at breakpoints and locations of the breakpoint, are in fact immediately available as soon as the system of equation presented for the piecewise linear lower bound is solved (Table \ref{tab:parameters}). 

The upper bound presented is closely related to a well-known inequality from stochastic programming, see e.g. \cite{citeulike:695971}, p. 168, \cite{citeulike:12355817}, p. 316, and \cite{citeulike:2516312}, pp. 291-293. As pointed out in \cite{citeulike:695971}, p. 168, Edmundson-Madanski's upper bound can be seen as a bound where the original distribution is replaced by a two point distribution and the problem itself is unchanged, or it can be viewed as a bound where the distribution is left unchanged and the original function is replaced by a linear affine function represented by a straight line. The above discussion clearly demonstrates the dual nature of this upper bound.

\subsection{Normal distribution}

We will next discuss the system of equations that leads to an optimal partitioning for the case of a standard Normal random variable $Z$. This partitioning leads to a piecewise linear approximation that is, in fact, easily extended to the general case of a normally distributed variable $\zeta$ with mean $\mu$ and standard deviation $\sigma$ via Lemma \ref{lem:folf_std_norm}. Also for this second approximation this equation suggests that the error is independent of $\mu$ and proportional to $\sigma$.

Consider a partitioning for the domain of $x$ in $\widehat{\mathcal{L}}(x,\omega)$ into $N+2$ adjacent regions $\mathcal{D}_i=[a_i,b_i]$, where $i=0,\ldots,N+1$. From Theorem \ref{thm:approx_err_symmetric_ub}, if $N$ is odd, then $b_{\lceil N/2\rceil}=0$ and $b_{i}=-b_{N+1-i}$, if $N$ is even, then $b_{i}=-b_{N+1-i}$. Also in this case, we shall use Lemma \ref{thm:compl_fol_closed} for expressing $\widehat{\mathcal{L}}(x,Z)$, and we will exploit the close connections between finding a local minimum and solving a set of nonlinear equations. We will therefore use the Gauss-Newton method to minimize the following sum of squares
\[\sum_{k=1}^{N+1}(e_0-e_k)^2\]
This minimisation problem can be solved by software packages such as Mathematica (see \texttt{NMinimize}). 

\subsubsection{Numerical examples}

A two-segment piecewise linear upper bound for the complementary first order loss function of a standard Normal random variable $Z$ is shown in Fig. \ref{fig:piecewise-2-ub}. This bound has been obtained, under the minimax criterion previously described, by considering a single breakpoint in the domain, i.e. $x=0$. Of course, the maximum error of this piecewise linear approximation occurs for $x=\pm\infty$ and it is equal to $1/\sqrt{2\pi}$.
\begin{figure}[h!]
\centering
\includegraphics[type=eps,ext=.eps,read=.eps,width=0.8\columnwidth]{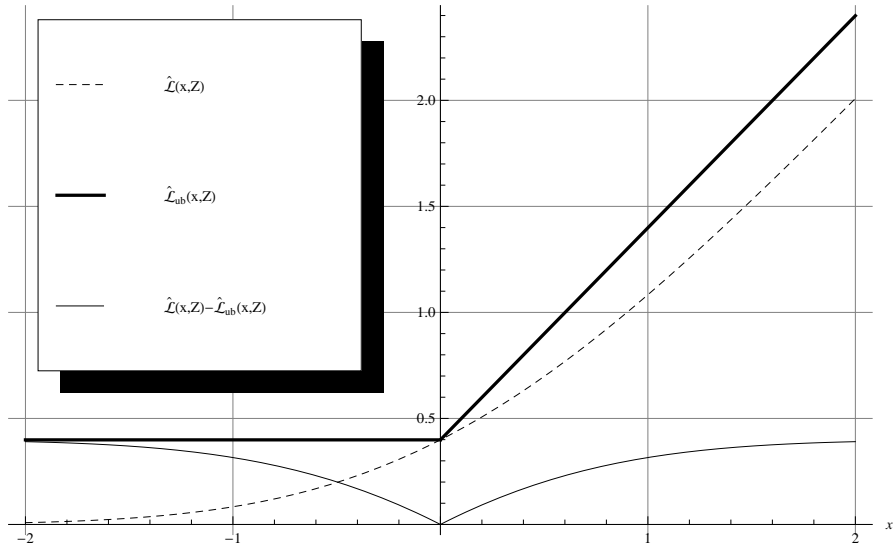}
\caption{two-segment piecewise linear upper bound for $\widehat{\mathcal{L}}(x,Z)$}
\label{fig:piecewise-2-ub}
\end{figure}   
It is easy to observe that this upper bound can be obtained by adding to the classical Jensen's lower bound presented in Fig. \ref{fig:piecewise-2} a constant value equal to its maximum approximation error, i.e. $1/\sqrt{2\pi}$.


We next present a more interesting case, in which the domain has been split into five regions (Fig. \ref{fig:piecewise-5-ub}). 
\begin{figure}[h!]
\centering
\includegraphics[type=eps,ext=.eps,read=.eps,width=0.8\columnwidth]{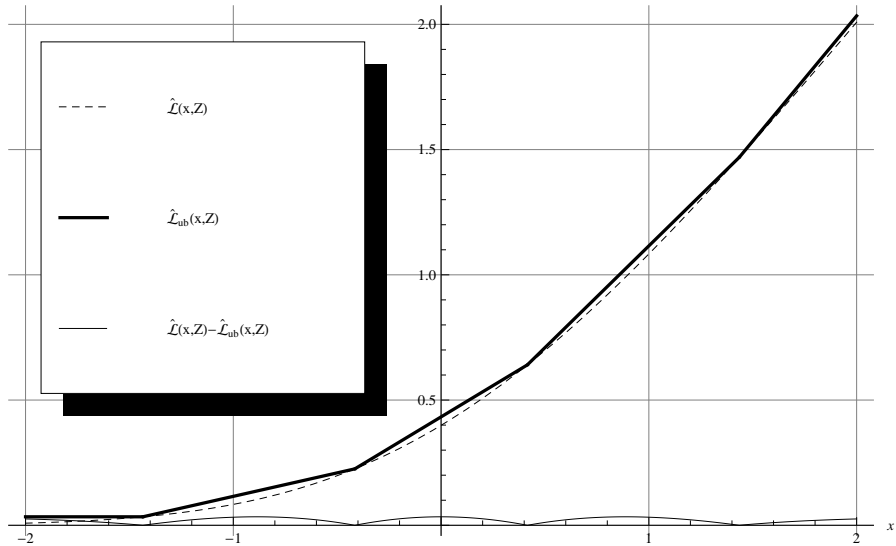}
\caption{five-segment piecewise linear upper bound for $\widehat{\mathcal{L}}(x,Z)$}
\label{fig:piecewise-5-ub}
\end{figure}   
Breakpoints are positioned at $x\in\{\pm1.43535,\pm0.415223\}$. These were the locations at which the maximum error, i.e. 0.0339052, was observed in Fig. \ref{fig:piecewise-5}. Also in this case, the five-segment piecewise linear upper bound can be obtained by adding to the five-segment piecewise Jensen's lower bound a value equal to its maximum approximation error.

Finally, in Fig. \ref{fig:piecewise-5-scaled-ub}, we show an example in which we exploited Lemma \ref{lem:folf_std_norm} to obtain, from the approximation presented in Fig. \ref{fig:piecewise-5-ub}, the five-segment piecewise linear upper bound for $\widehat{\mathcal{L}}(x,\zeta)$, where $\zeta$ is a normally distributed random variable with mean $\mu=20$ and standard deviation $\sigma=5$. The maximum error is $\sigma0.0339052$ and it is observed at $x\in\{\pm\infty,\sigma(\pm0.886942)+\mu, \mu\}$.
\begin{figure}[h!]
\centering
\includegraphics[type=eps,ext=.eps,read=.eps,width=0.8\columnwidth]{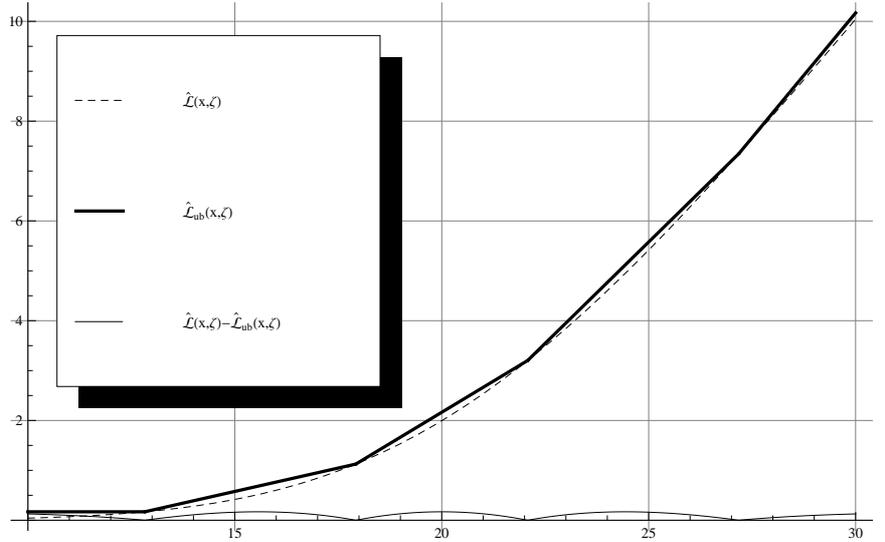}
\caption{five-segment piecewise linear upper bound for $\widehat{\mathcal{L}}(x,\zeta)$, where $\mu=20$ and $\sigma=5$}
\label{fig:piecewise-5-scaled-ub}
\end{figure}   

\section{Conclusions}

We summarised a number of distribution independent results for the first order loss function and its complementary function. We then focused on symmetric distributions and on normal distributions; for these we discussed ad-hoc results. To the best of our knowledge a comprehensive analysis of results concerning the first order loss function seems to be missing in the literature. The first contribution of this work was to fill this gap in the literature. Based on the results discussed, we developed effective piecewise linear approximation strategies based on a minimax framework. This is the second contribution of our work. More specifically, we developed piecewise linear upper and lower bounds for the first order loss function and its complementary function. These bounds rely on constant parameters that are independent of the means and standard deviation of the normal distribution considered. We discussed how to compute optimal parameters that minimise the maximum approximation error and we also provided a table with pre-computed optimal parameters for piecewise bound with up to eleven segments. These bounds can be easily embedded in existing MILP models. 

\section*{Acknowledgements}

INSERT ACKNOWLEDGEMENTS HERE

\bibliographystyle{plain}
\bibliography{piecewise}	

\begin{thebibliography}{10}

\bibitem{Axsater2006}
Sven Axsater.
\newblock {\em Inventory control}.
\newblock Springer Verlag, 2006.

\bibitem{citeulike:2516312}
John~R. Birge and Fran\c{c}ois Louveaux.
\newblock {\em Introduction to Stochastic Programming (Springer Series in
  Operations Research and Financial Engineering)}.
\newblock Springer, corrected edition, July 1997.

\bibitem{citeulike:12461170}
Wlodzimierz Bryc.
\newblock A uniform approximation to the right normal tail integral.
\newblock {\em Applied Mathematics and Computation}, 127(2-3):365--374, April
  2002.

\bibitem{citeulike:12317680}
Steven~K. De~Schrijver, El-Houssaine Aghezzaf, and Hendrik Vanmaele.
\newblock Double precision rational approximation algorithm for the inverse
  standard normal first order loss function.
\newblock {\em Applied Mathematics and Computation}, 219(3):1375--1382, October
  2012.

\bibitem{citeulike:12355817}
Karl Frauendorfer and Michael Sch\"{u}rle.
\newblock Stochastic linear programs with recourse and arbitrary multivariate
  distributions.
\newblock In Christodoulos~A Floudas and Panos~M Pardalos, editors, {\em
  Encyclopedia of Optimization}, pages 2488--2493. Springer US, 2001.

\bibitem{citeulike:695971}
Peter Kall and Stein~W. Wallace.
\newblock {\em Stochastic Programming (Wiley Interscience Series in Systems and
  Optimization)}.
\newblock John Wiley \& Sons, August 1994.

\bibitem{citeulike:12461174}
Jean~M. Linhart.
\newblock Algorithm 885: Computing the logarithm of the normal distribution.
\newblock {\em ACM Trans. Math. Softw.}, 35(3), October 2008.

\bibitem{citeulike:1194520}
Renata Mansini, W{\l}odzimierz Ogryczak, and Maria~Grazia Speranza.
\newblock Conditional value at risk and related linear programming models for
  portfolio optimization.
\newblock {\em Annals of Operations Research}, 152(1):227--256, July 2007.

\bibitem{shore1982}
Haim Shore.
\newblock Simple approximations for the inverse cumulative function, the
  density function and the loss integral of the normal distribution.
\newblock {\em Journal of the Royal Statistical Society. Series C (Applied
  Statistics)}, 31(2):pp. 108--114, 1982.

\bibitem{spp98}
Edward~A. Silver, David~F. Pyke, and Rein Peterson.
\newblock {\em {Inventory Management and Production Planning and Scheduling}}.
\newblock John-Wiley and Sons, New York, 1998.

\bibitem{citeulike:12461167}
Gary~R. Waissi and Donald~F. Rossin.
\newblock A sigmoid approximation of the standard normal integral.
\newblock {\em Applied Mathematics and Computation}, pages 91--95, June 1996.

\bibitem{zelen64}
Marvin Zelen and Norman~C. Severo.
\newblock Probability functions.
\newblock In Milton Abramowitz and Irene~A Stegun, editors, {\em Handbook of
  Mathematical Functions}, volume~5 of {\em Applied Mathematics Series}, pages
  925--995. GPO, 1964.

\end{thebibliography}

\end{document}